\documentclass{amsart}
\usepackage[utf8]{inputenc}
\usepackage{footnote}
\usepackage[backend=biber, style=alphabetic]{biblatex}
\usepackage{graphicx}
\usepackage{caption}
\usepackage{mathtools}  
\usepackage{mathrsfs}   
\usepackage{bm}         
\usepackage{tikz-cd}
\usepackage{enumitem}
\usepackage{amssymb}
\usepackage{stmaryrd}
\usepackage[T1]{fontenc}
\usepackage{CJKutf8}        % For Japanese text support in pdfLaTeX
\usepackage{mathtools}  
\usepackage{extarrows}
\usepackage{mathrsfs}
\usepackage[mathcal]{eucal}
\usepackage{bbm}

\usepackage[dvipsnames,svgnames,table]{xcolor}
\usepackage{todonotes}
\usepackage[unicode=true,pdfusetitle,
 bookmarks=true,bookmarksnumbered=false,bookmarksopen=false,
 breaklinks=false,pdfborder={0 0 0},pdfborderstyle={},backref=false,colorlinks=true]{hyperref}
 
\hypersetup{
    colorlinks, linkcolor=OliveGreen,
    citecolor=OliveGreen, urlcolor=OliveGreen
}
\usepackage[nameinlink,capitalise,noabbrev]{cleveref}

\crefalias{theorem}{theorem}
\crefalias{proposition}{proposition}
\crefalias{corollary}{corollary}
\crefalias{lemma}{lemma}
\crefalias{example}{example}
\crefalias{remark}{remark}

\theoremstyle{definition}
\newtheorem{definition}{Definition}[section]
\newtheorem{example}[definition]{Example}

\newtheorem{remark}[definition]{Remark}

\theoremstyle{plain}
\newtheorem{theorem}[definition]{Theorem}
\newtheorem{proposition}[definition]{Proposition}
\newtheorem{corollary}[definition]{Corollary}

\newtheorem{lemma}[definition]{Lemma}

\newcommand{\adjoint}{\ \substack{ \longrightarrow \\[-1em] \longleftarrow}\ }

\begin{document}
\title{A Super Version of a theorem of Fricke-Klein}
\author{Marcel Dang}
\begin{abstract}
    In this paper, we study the character variety of the supergroup $\operatorname{OSp}(1|2)$ in the specific case of the free group on two letters as a first step towards an algebraic treatment of character varieties of supergroups. One can interpret this as an analogue of a result of Fricke and Klein, who studied this in the case of $\operatorname{SL}(2,\mathbb{C})$. We prove that the superring $\mathbb{C}[\mathrm{OSp}(1|2)^2]^{\mathrm{OSp}(1|2)}$ is an even superring and hence is an ordinary commutative ring. We also discuss a few geometric consequences at the end.
\end{abstract}
\maketitle 
\tableofcontents
\setcounter{tocdepth}{1}
\section{Introduction}
\subsection{Motivation}
A result by Fricke and Klein \cite{fricke1897vorlesungen} states that the $SL(2,\mathbb{C})$-invariant functions on $SL(2,\mathbb{C}) \times SL(2,\mathbb{C})$ are given by the ring 
\begin{align}
    \mathbb{C}[\operatorname{tr}(A),\operatorname{tr}(B),\operatorname{tr}(AB)].
\end{align}
In the language of geometric invariant theory, we can identify the categorical quotient of the product variety $\mathrm{SL}_2(\mathbb{C}) \times \mathrm{SL}_2(\mathbb{C})$ by conjugation as affine space, i.e.
\begin{center}
    \begin{tikzcd}
    \mathrm{SL}_2(\mathbb{C}) \times \mathrm{SL}_2(\mathbb{C}) \arrow[r, "\pi"] & \operatorname{Spec}(\mathbb{C}[\operatorname{tr}(A),\operatorname{tr}(B),\operatorname{tr}(AB)]) \simeq \mathbb{C}^3.
    \end{tikzcd}
\end{center}
Another geometric reinterpretation of this result is  that the character variety of the free group on two letters and $SL(2,\mathbb{C})$ is isomorphic to $\mathbb{C}^3$, i.e. 
\begin{align}
    \operatorname{Hom}(F_2,SL(2,\mathbb{C}))/\!/SL(2,\mathbb{C}) \simeq \mathbb{C}^3.
\end{align} 
In this paper we establish a variant of the theorem of Fricke and Klein for a supergeometric extension of $\mathrm{SL}_2(\mathbb{C})$, specifically the complex supergroup $\mathrm{OSp}(1|2)$. The motivation of this result comes from the fact that $\operatorname{OSp}(1|2)$ is the super generalization of $\operatorname{SL}(2,\mathbb{C})$ and takes a similar place in the theory of super Riemann surfaces, as $\operatorname{SL}(2,\mathbb{C})$ does in the classical theory. On the real analytic side of this story, the character variety is related to the super Teichmüller space, where each connected component of the space is homeomorphic to $\mathbb{R}^{6g-6|4g-4}/\mathbb{Z}_2$ \cite{bouschbacher2013shear} \cite{penner2019decorated}, where $\mathbb{Z}_2$ only acts on the odd coordinates. \\
Initially, one might guess that an analogous result holds for $\mathrm{OSp}(1|2)$, making its character variety for $F_2$ isomorphic to $\mathbb{C}^{3|2}$, or isomorphic to $\mathbb{C}[X,Y,Z|\Theta,\Psi]$ as a ring. This is not the case and we will show that the character variety for the free group on two letters $F_2$ and the supergroup $\mathrm{OSp}(1|2)$ fails to be a superscheme. 
\begin{theorem}[{\cref{main}, \cref{atleast4}}]
The algebra of $\operatorname{OSp}(V)$-invariant polynomial functions (by conjugation) on $\operatorname{OSp}(V) \times \operatorname{OSp}(V)$ is concentrated in degree 0. Moreover, it is not generated by $\mathrm{str}(A)$, $\mathrm{str}(B)$ and $\mathrm{str}(AB)$.
\end{theorem}
Hence, the ring of invariants $\mathbb{C}[\mathrm{OSp}(V)^2]^{\mathrm{OSp}(V)}$ is an ordinary commuative ring, as opposed to the initial guess $\mathbb{C}[X,Y,Z|\Theta,\Psi]$ which contains degree 1 elements.
To do this, we will extend notions of classical invariant theory to the super setting.
\subsection{Organization of the paper}
 In section 2 we will extend the necessary classical invariant theory to the super setting to prove our main result. In section 3 we prove our main result \cref{main} and give an example of a trace that is not generated by $\mathrm{str}(A)$, $\mathrm{str}(B)$ and $\mathrm{str}(AB)$. In section 4 we introduce the character variety and use it to give a geometric interpretation of the main result. Afterwards, we introduce the character stack and discuss how the notion of a stack fixes the ill-behaved notion of character variety in the supergeometric setting.
\subsection{Notation and Conventions}
We exclusively work over algebraically closed fields of characteristic $0$, which we denote by $k$. For the proof of the main theorem, we will specialize to the complex numbers $\mathbb{C}$. 
\begin{itemize}
    \item We denote the category of super vector spaces over $k$, i.e. the category of $\mathbb{Z}_2$-graded vector spaces that uses the Koszul sign in the braiding of the symmetric monoidal structure by $\mathsf{sVect}$. 
    \item We denote the category of superrings, i.e. the category of commutative algebra objects in the category of super vector spaces, by $\mathsf{sRing}\coloneqq \mathsf{CAlg}(\mathsf{sVect})$.
    \item For $V \in \mathsf{sVect}$, we call its degree 0 part  even and denote it by $V_0$, and we call its degree 1 part odd and denote it by $V_1$. A similar convention holds for a superring $R$.
    \item We denote $\Lambda^i\coloneqq \wedge(e_1,...,e_i)$ the free superalgebra generated by a super vector space $V$ of dimension $0|i$.
    \item We denote the category of finite dimensional rational representations by $\mathsf{rep}(G)$ and the category of possibly infinite dimensional rational representations by $\mathsf{Rep}(G)$.
    \item We denote the category of superschemes by $\mathsf{sSch}$, where by superscheme we mean a Zariski sheaf $X:\mathsf{sRing}\to \mathsf{Set}$ that admits a Zariski atlas, i.e. for a scheme $X$ there exists a collection of morphisms from affine schemes $f_i:S_i\to X$ such that each $f_i$ is an open immersion.
    \item We call a group object in the category of superschemes an algebraic supergroup, which in the affine case is equivalent to a Hopf superalgebra.
    \item Adjunctions are denoted by $L: \mathcal{C} \adjoint \mathcal{D}:R$, where $L$ is the left adjoint and $R$ is the right adjoint. 
    \item The bosonic reduction $(-)_\mathrm{red}$ is the right adjoint of the fully faithful inclusion  $\iota:\mathsf{Sch}\to \mathsf{sSch}$: $(-)_{\mathrm{red}}$.
    \item The invariants functor $(-)^G$ of an algebraic supergroup $G$ is the right adjoint of the functor $\mathrm{triv}$ that endows a super vector space with the trivial group action, i.e. $\mathrm{triv}:\mathsf{sVect} \adjoint\mathsf{Rep}(G): (-)^G$. Since $\mathrm{triv}$ is strong symmetric monoidal, the adjunction restricts to commutative algebra objects, which we are going to denote by the same symbols. In particular, this makes $(-)^G$ a well-defined operation on superrings.
\end{itemize}
Most of the theory of commutative algebra carries over to the $\mathbb{Z}_2$-graded setting mutatis mutandis. For the interested reader, we recommend Westra's thesis \cite{westra2009superrings}.
\subsection{Acknowledgements}
The author would like to thank his advisor Paul Norbury for his constant guidance and confidence in his work, as well as numerous suggestions on earlier versions of this paper. The author also thanks Gustav Berth, Miles Koumouris and Lior Yanovski for useful discussion. \\
I am supported by the University of Melbourne Graduate Research Scholarship.
\subsection{AI Usage} The author has used AI for literature research, assistance for checking mathematical arguments, proofreading and writing. The author takes full responsibility for the final statements and proofs.
\section{Preliminaries}
\subsection{\texorpdfstring{$\operatorname{OSp}(1|2)$}{V}}
Let us define the supergeometric counterpart of the group $\mathrm{SL}(2,\mathbb{C})$ in the theory of Riemann surfaces.
\begin{definition}
    The orthosymplectic supergroup $\mathrm{OSp}(1|2)$ is represented by the group functor 
    \begin{align}
    \begin{split}
               \mathrm{OSp}(1|2): \mathsf{sRing} &\longrightarrow \mathsf{Grp} \\
       R &\longmapsto \left\{ \left.      
       M= \begin{pmatrix*}
           a & b & \alpha \\
           c & d & \beta \\
           \gamma & \delta & f
       \end{pmatrix*} \right| M^{st}JM = J
       \right\},
    \end{split}
    \end{align}
    such that $a,b,c,d,f \in R_0$ and $\alpha, \beta, \gamma, \delta \in R_1$ and $\mathrm{Ber}(M)=1$, where 
    \begin{align}
       J= \begin{pmatrix*}
            0 & -1 & 0 \\
            1 & 0 & 0 \\
            0 & 0 & 1
        \end{pmatrix*}
    \end{align}
    and $M^{st}$ denotes the super transpose of $M$.
\end{definition}
Evaluating the group functor of $\mathrm{OSp}(1|2)$ and a vector space $V$ in $\Lambda$-points, we can give a more familiar definition as a set.
\begin{definition}[\cite{lehrer2017first}]
   Let $V$ the functor of points of a super vector space and denote its $\Lambda$-points by $V(\Lambda)$, where $\Lambda =\mathsf{colim}_i \Lambda_i$. Furthermore, let be $B$ be an even non-degenerate bilinear form then the group of orthosymplectic transformation with respect to $B$ is defined as
    \begin{align}
        \operatorname{OSp}(1|2) \coloneqq  \{g \in \operatorname{GL}(V(\Lambda))| B(gv,gw)=B(v,w) \ \forall v,w \in V(\Lambda)\}.
    \end{align}
\end{definition}
The group law is given by ordinary matrix multiplication.
Equivalently, we can also describe it as a group preserving a supersymplectic structure $\omega$ on $k^{2|1}$ \cite{norbury2020enumerative}. We note that the class of linearly reductive algebraic supergroups is small class, which is a result originally due to  Djokovi\'c and Hochschild \cite{djokovic1976semisimplicity}.
\begin{theorem}[cf. \cite{sherman2021geometricproofsclassificationalgebraic}] \label[theorem]{semisimple}
    Let $G$ be a connected algebraic supergroup. Then $\mathsf{rep}(G)$ is semisimple if and only if
    \begin{align}
        G = K \times \operatorname{OSp}(1|2n_1) \times ... \times \operatorname{OSp}(1|2n_k)
    \end{align}
    for a reductive algebraic group $K$ and positive integers $n_1,...n_k$.
\end{theorem}
    Let us now prove that $(-)^G$ is exact. Since $\mathsf{rep}(G)$ is semisimple, by \cref{semisimple}, $G$ is linearly reductive, and one can admit the following lemma:
\begin{lemma} \label[lemma]{reynolds} 
Let $G$ be a linearly reductive algebraic supergroup that acts rationally on a $k$-super vector space $V$ and let $V^G$ be the subspace of invariant vectors. Then there exists a uniquely defined $G$-equivariant linear operator 
    \begin{align}
        R: V \longrightarrow V 
    \end{align}
    projecting $V$ onto $V^G$.
\end{lemma}
\begin{proof}
    Since $V$ is a rational representation, it is locally finite. Hence, we find that $V$ is a filtered colimit of finite dimensional representations $W_i$, i.e. 
    \begin{align}
        V \simeq \mathsf{colim}_i W_i.
    \end{align}
    Note that the fixed point functor is the right adjoint to the trivial action functor
    \begin{align}
            \mathrm{triv}: \mathsf{sVect} \adjoint\mathsf{Rep}(G): (-)^G,
    \end{align}
    and observe that the categories participating in the adjunction are (locally) presentable and thus,
    $(-)^G$ is accessible, that is $(-)^G$ preserves filtered colimits, by the adjoint functor theorem \cite[Theorem 4.7.17.]{riehl2017category}. As a consequence we obtain 
    \begin{align}
         V^G \simeq \mathsf{colim}_i (W_i)^G.
    \end{align}
    By linear reductivity of $G$ the finite dimensional vector spaces $W_i$ admit a direct sum decomposition as representations:
    \begin{align}
        W_i \simeq W_i^G \oplus W'_i,
    \end{align}
     yielding $G$-equivariant projection maps $R_{W_i}: W_i \to (W_i)^G$. They assemble into a compatible system of maps inducing, by universal property, a map
    \begin{align}
      R':  V \longrightarrow V^G,
    \end{align}
    which we extend to a map $R: V \to V$, by composing with the counit map at $V$, i.e. the natural inclusion map $V^G \to V$.
\end{proof}
\begin{proposition} \label[proposition]{fixedexact}
    Let $G$ be an algebraic supergroup such that $\mathsf{rep}(G)$ is semisimple, then the invariants functor
    \begin{align}
        (-)^G: \mathsf{Rep}(G) \longrightarrow \mathsf{sVect},
    \end{align}
    is exact, i.e. it preserves finite limits and colimits.
\end{proposition}
\begin{proof}
    We prove the claim by reducing it to showing that $(-)^G$ preserves surjections, which then follows from \cref{reynolds}. The functor $(-)^G$ fits into the following adjunction:
    \begin{align}
        \mathrm{triv}: \mathsf{sVect} \adjoint\mathsf{Rep}(G): (-)^G,
    \end{align}
    where $\mathrm{triv}$ is the functor that endows a super vector space with the trivial $G$-action. One observes that $(-)^G$ is a right adjoint and hence preserves all limits, thus it suffices to show that $(-)^G$ preserves finite colimits. Preserving finite colimits is equivalent to showing that finite coproducts and coequalizers are preserved, which is a consequence of \cite[\href{https://stacks.math.columbia.edu/tag/04AS}{Tag 04AS}, Lemma 4.18.7]{stacks-project}. Finite coproducts are given by direct sums in $\mathsf{Rep}(G)$, which are preserved by $(-)^G$, as direct sums are both finite coproducts and finite products, so it remains to show that coequalizers are preserved. Every coequalizer is an epimorphism, which is the dual claim of \cite[Proposition 9.1.5.]{barr1990category}, thus we only need to show that $(-)^G$ preserves surjections. By definition, a surjective map of representations is surjective if and only if it is surjective on the underlying vector spaces. The map $R$, defined in \cref{reynolds} is a projection, which is known to preserve surjections, hence the claim follows.
\end{proof}
\begin{remark}
    The functor $\mathrm{triv}: \mathsf{sVect} \to \mathsf{Rep}(G)$ is strong monoidal and hence, the adjunction lifts to the categories of commutative algebra objects 
    \begin{align}
        \mathrm{triv}_{\mathrm{alg}}: \mathsf{sRing} \adjoint\mathsf{sRing}^G: (-)^G_{\mathrm{alg}}.
    \end{align}
    However, the functor $(-)^G_{\mathrm{alg}}$ does not preserve finite colimits generally, as opposed to $(-)^G$. In particular, finite coproducts are not preserved, which are given by the tensor product. One can consider the example $\mathbb{G}_m$ acting on $A=k[X]$ of weight $1$ and acting on $B=k[Y]$ with weight $-1$. Then observe that 
    \begin{align}
        k[X,Y]^{\mathbb{G}_m} = k[XY] \neq k[X]^{\mathbb{G}_m} \otimes_{k^{\mathbb{G}_m}} k[Y]^{\mathbb{G}_m}  = 0.
    \end{align}
\end{remark}
\subsection{Invariant Theory}
In this section, we introduce the necessary tools from classical invariant theory and adapt them to the $\mathbb{Z}_2$-graded case. Most of these results will generalize without adapting the proof and can be skipped for the experienced reader. We follow \cite{kraft1996invariant} and \cite{mumford1994geometric}.\\
We establish that polarization and restitution works just as in the purely bosonic case. For the reader unfamiliar with classical invariant theory, polarization allows us to transform the polynomial into a multilinear expression, where we already know the invariants and restitution is the inverse process. We combine them in the following way 
\begin{center}
    \begin{tikzcd}
        \text{Polarization} \arrow[r] & \text{Computing Invariants} \arrow[r] & \text{Restitution}
    \end{tikzcd}
\end{center}
which allows us to compute the polynomial invariants, which arise out of the multilinear ones. To be more precise, we will show that 
\begin{center}
    \begin{tikzcd}
        \bigoplus_d k[V]_d \arrow[r,"\mathcal{P}"] & \bigoplus_d k[V^d] \arrow[r,"\mathcal{R}"] & \bigoplus_d k[V]_d
    \end{tikzcd}
\end{center}
is $G$-equivariant, where the maps will be explained shortly. We can see that one can simply chase an invariant polynomial $f \in k[V]$ through this construction, find an explicit form in terms of multilinear expressions and then recover the original polynomial in terms of the multilinear invariants. This strategy shows why it is enough to be able to compute multilinear invariants, which is simpler than determining polynomial invariants directly.
\begin{definition}
    Consider a direct sum $V=V_1 \oplus ... \oplus V_r$ of finite dimensional super vector spaces. A function $f \in k[V_1 \oplus ... \oplus V_r]$ is called multihomogeneous of degree $h=(h_1,...,h_r)$ if $f$ is homogeneous of degree $h_i$ in $V_i$, i.e. for all $v_i \in V_i$ and $t_1,...,t_r \in k$ we have 
    \begin{align}
        f(t_1v_1,...,t_rv_r) = t_1^{h_1}\cdot...\cdot t_r^{h_r}f(v_1,...v_r).
    \end{align}
    One can also call $h$ the multidegree of $f$.
\end{definition}
Every polynomial function $f$ is in a unique way a sum of multihomogeneous functions $f= \sum f_h$. The $f_h$ are usually called multihomogeneous components of $f$. This gives rise to a decomposition 
\begin{align}
    k[V_1\oplus...\oplus V_r] = \bigoplus_{ h \in \mathbb{N}^r} k[V_1,...,V_r]_h,
\end{align}
where $k[V_1,...,V_r]_h$ denotes the subspace of multihomogeneous functions of degree $h$.
\begin{remark}
    This is a different grading from the natural $\mathbb{Z}_2$-grading of our superring.
\end{remark}
It is clear that the multidegree is preserved under the action of $\operatorname{GL}(V_1) \times ... \times \operatorname{GL}(V_r)$. Furthermore, if $V_1,...V_r$ are representations, then we also have a natural grading on the ring of invariants 
\begin{align}
    k[V_1,...,V_r]^G = \bigoplus_{ h \in \mathbb{N}^r} k[V_1,...,V_r]_h^G.
\end{align}
\begin{definition}
    Let $V$ be a finite dimensional $k$-vector space. Identify $k[V]= \bigoplus_d k[V]_d$, where $k[V]_d$ denotes the space of homogeneous polynomials of degree $d$. Let $f \in k[V]_d$ and $v = \sum_i^dt_iv_i$ for $v_i \in V$. We define 
    \begin{align*}
       f(\sum_it_iv_i)= \sum_{s_1+...+s_d=d}t_1^{s_1}...t_d^{s_d}f_{s_1,...,s_d}(v_1,...,v_d),
    \end{align*}
    where $f_{s_1,...,s_d} \in k[V^d]$ are well-defined and multihomogeneous of multidegree $(s_1,...,s_d)$. We call $f_{1,...,1} \in k[V^d]$ the (full) polarization of $f$ denoted by $\mathcal{P}f$.
\end{definition}
\begin{lemma}\label[lemma]{propertiesof}
    The linear operator $\mathcal{P} \colon k[V]_d \to k[V^d]_{(1,1,\ldots,1)}$ has the following properties:
\begin{enumerate}[label=(\alph*)]
  \item $\mathcal{P}$ is $\mathrm{GL}(V)$-equivariant;
  \item $\mathcal{P}f$ is symmetric;
  \item $\mathcal{P}f(v, v, \ldots, v) = d! \, f(v)$.
\end{enumerate}
\end{lemma}
\begin{proof}
    For equivariance consider the diagram 
    \begin{center}
        \begin{tikzcd}
            k[V]_d \arrow[r, "\mathcal{P}"] \arrow[d, "g \cdot"] & k[V^d] \arrow[d, "g \cdot"] \\
            k[V]_d \arrow[r, "\mathcal{P}"] & k[V^d]
        \end{tikzcd},
    \end{center}
    and let 
    \begin{align}
        f(t_1v_1+...+t_dv_d) = \sum t_1^{s_1}\cdot ... \cdot t_d^{s_d} f_{s_1...s_d}(v_1,...v_d).
    \end{align}
    We start chasing 
    \begin{center}
        \begin{tikzcd}
            \sum t_1^{s_1}\cdot ... \cdot t_d^{s_d} f_{s_1...s_d}(v_1,...v_d) \arrow[r, mapsto] \arrow[d, "g\cdot"]&  f_{1,...,1}(v_1,...v_d) \arrow[d, "g \cdot"]\\
            \sum t_1^{s_1}\cdot ... \cdot t_d^{s_d} f_{s_1...s_d}(g^{-1}v_1 ,...,g^{-1}v_d ) \arrow[r,mapsto] & f_{1,...,1}(g^{-1}v_1 ,...g^{-1}v_d)
        \end{tikzcd}.
    \end{center}
    $\mathcal{P}$ is also obviously symmetric. It remains to show $(c)$. Let us calculate 
    \begin{align}
        f((\sum_i t_i)v) = (\sum_i t_i)^df(v) = (d!t_1\cdot...\cdot t_d + ...) f(v).
    \end{align}
    Looking at the first factor the claim follows.
\end{proof}
\begin{definition}
    For a multilinear $F \in k[V^d]$ the homogeneous polynomial $\mathcal{R}F$ defined by $\mathcal{R}F(v) := F(v, v, \ldots, v)$ is called the (full) restitution of $F$.
\end{definition}
By construction we also have the following result.
\begin{proposition}
   Let $V$ be a finite dimensional representation of a supergroup $G$. Then every homogeneous invariant $f \in k[V]^G$ of degree $d$ is the full restitution of a multilinear invariant $F \in k[V^d]^G$.
\end{proposition}
\begin{proof}
    Again, $\mathcal{R} \colon k[V^d]_{(1, \ldots, 1)} \to k[V]_d$ is a linear $\mathrm{GL}(V)$-equivariant operator and we have $\mathcal{R} \mathcal{P} f = d! f$ by property (c) of \cref{propertiesof}. In fact, $f$ is the full restitution of $\frac{1}{d!}\mathcal{P}f$, which is a multihomogeneous invariant.
\end{proof}
Let us state the first fundamental theorem for $GL(V)$:
\begin{theorem}[cf. {\cite{deligne2018first}}]
    Let $V$ be a finite dimensional super vector space. The multilinear invariants under the natural $GL(V)$-action are given by 
    \begin{align}
    \begin{split}
          V^{\otimes N} \otimes (V^\vee)^{\otimes N} &\longrightarrow \mathbb{C}  \\
        v_1 \otimes ... \otimes v_N \otimes \phi_1 \otimes ... \otimes \phi_N &\longmapsto \prod_i \langle v_i, \phi_{\sigma (i)}\rangle
    \end{split}
    \end{align}
    and all its transforms under the symmetric group
\end{theorem}
It remains to show, that the pairings are still given by traces, as in \cite{procesi1976invariants}. 
\begin{proposition}\label[proposition]{traces}
    Let V be a $p|q$ dimensional super vector space and $V^\vee$ its dual. We take $\langle \cdot, \cdot \rangle: V \otimes V^\vee \rightarrow k$, then
    \begin{align*}
        \langle v,\phi \rangle = \operatorname{str}( v \otimes \phi )
    \end{align*}
\end{proposition}
\begin{proof}
    We prove this by calculation. Take a homogeneous basis $\left\{e_j\right\}$ for $V$ and let $\left\{\varphi_j\right\}$ be its dual basis. We take a vector $v \otimes \varphi \in V \otimes V^\vee$ and compute the coefficients of the associated matrix $A$. Let 
\begin{align*}
    v&=\sum_i a_i e_i \quad, \ \text{such that}\ \forall |e_i|=\{0,1\}: \  a_i \in \Lambda_{\{0,1\}} \\
    \phi&=\sum_j b_j \varphi_j \quad, \ \text{such that}\ \forall |\varphi_j|=\{0,1\}: \  b_j \in \Lambda_{\{0,1\}} 
\end{align*}
and we compute 
    \begin{align*}
      v \cdot \phi =  \sum_{ij} a_i b_j e_i \otimes \varphi_j
    \end{align*}
    on the other hand we can also compute 
    \begin{align*}
        \langle v, \phi \rangle = \phi \cdot v = \sum_{i} (-1)^{|e_i||\varphi_i|}a_ib_i . 
    \end{align*}
    One then observes that
    \begin{align*}
        \operatorname{str}(A) = \langle v, \phi \rangle.
    \end{align*}
\end{proof}
\begin{lemma}\label{multilinear-invariants}
    Let $V$ be a super vector space. Then the space of multilinear functions on $V^m \oplus (V^\vee)^m$ is isomorphic to the space of multilinear functions on $\operatorname{End}(V)^m$
\end{lemma}
\begin{proof}
    The space of multilinear functions on $V^m \oplus (V^\vee)^m$ is given by 
    \begin{align}
       ( V^{\otimes m} \otimes (V^\vee)^{\otimes m})^\vee \simeq ((V \otimes V^\vee)^{\otimes m})^\vee \simeq (\operatorname{End}(V)^{\otimes m})^\vee.
    \end{align}
    Now, we observe that the right hand side is the space of multilinear functions on $\operatorname{End}(V)^m$.
\end{proof}
\begin{remark}
   These spaces naturally admit a $\operatorname{GL}(V)$-action, and considering them as $\operatorname{GL}(V)$-representations, it is also easy to see, that these identifications are $GL(V)$-equivariant isomorphisms.
\end{remark}
We can now translate the multilinear invariant in vector form into the multilinear invariant in matrix form. Consider $\operatorname{End}(V)^{\otimes i}$ and choose $\sigma \in Sym_i$. Decompose $\sigma $ into disjoint cycles 
\begin{align}
    \sigma = (i_1i_2...i_k)(j_1j_2...j_n)...(t_1...t_e)
\end{align}
including the ones of length 1.
\begin{theorem} \label{invariants-are-supertraces}
    Given $A_1,...,A_i \in \operatorname{End}(V)$, we have: 
    \begin{align}
        \mu_\sigma (A_1 \otimes ... \otimes A_i) = \operatorname{str}(A_{i_1}... A_{i_k})\operatorname{str}(A_{j_1}... A_{j_n})\operatorname{str}(A_{t_1}... A_{t_e}).
    \end{align}
\end{theorem}
\begin{proof}
    Since both sides are multilinear, it is enough to prove it for the case where $A_i= \phi_i \otimes v_i$.
    \begin{align}
        \mu_\sigma (A_1 \otimes ... \otimes A_i) &= \prod_i \langle \phi_{\sigma(i)} , v_i\rangle \\
        &=\langle \phi_{i_2}, v_{i_1} \rangle \langle \phi_{i_3}, v_{i_2}\rangle ... \langle \phi_{t_1}, v_{t_e}\rangle
    \end{align}

Consider, for instance, the product
\begin{align}
M &= \langle \varphi_{i_2}, v_{i_1} \rangle \langle \varphi_{i_3}, v_{i_2} \rangle \cdots \langle \varphi_{i_k}, v_{i_{k-1}} \rangle \langle \varphi_{i_1}, v_{i_k} \rangle.
\end{align}

One verifies immediately that
\begin{align}
\varphi_{i_1} \otimes v_{i_1} \cdot \varphi_{i_2} \otimes v_{i_2}, \ldots, \varphi_{i_k} \otimes v_{i_k}
&= \langle \varphi_{i_2}, v_{i_1} \rangle \langle \varphi_{i_3}, v_{i_2} \rangle \cdots \langle \varphi_{i_k}, v_{i_{k-1}} \rangle \varphi_{i_1} \otimes v_{i_k}.
\end{align}
Therefore,
\begin{align}
M = \mathrm{str}(A_{i_1} A_{i_2} \cdots A_{i_k}),
\end{align}
\end{proof}
\section{Main Result}
\subsection{Proof of Main Result}
For the proof of the main result let us admit the following theorem.
\begin{theorem}[\cite{deligne2018first}] \label[theorem]{FFT}
   Let $V$ be a finite dimensional super vector space with $dim(V)=m|2n$, then the multilinear $\operatorname{OSp}(V)$-invariant functions on $V^N$ for $N$ even are given by 
    \begin{align}
        \begin{split}
            V^{\otimes N} &\longrightarrow \mathbb{C} \\
        v_1 \otimes ... \otimes v_N &\longmapsto B(v_1,v_2)\cdot...\cdot B(v_{N-1},v_N),
        \end{split}
    \end{align}
    and all its permutations in $Sym_N$.
\end{theorem}
This was also proven by Lehrer and Zhang in \cite{lehrer2017first} with a different method of proof than in \cite{deligne2018first} using the language of Brauer diagrams. In \cite{lehrer2017first}, they also computed the polynomial invariants ones from the multilinear ones, which are given by the monomials in the multilinear invariants. 
\begin{theorem}[cf. {\cite{lehrer2017first}}] \label[theorem]{polynomialfft}
    Let $V$ be a finite dimensional super vector space with $dim(V)=m|2n$, then the polynomial $\operatorname{OSp}(V)$-invariants are polynomials in the variables $B(v_i,v_j) \coloneqq (i,j)$ $\forall 1 \leq i \leq j \leq m$
\end{theorem}
\begin{theorem}\label[theorem]{main}
 Let $V$ be a super vector space of dimension $1|2$ and let $\mathrm{OSp}(V)$ act on $\mathrm{OSp}(V)^2$ by conjugation. Then the ring $ \mathbb{C}[\mathrm{OSp}(V)^2]^{\operatorname{OSp}(V)}$ is even.
\end{theorem}
\begin{proof}
Observe that the superscheme $\mathrm{OSp}(V)^2$ is a subscheme of $\mathrm{End}(V)^2$, hence there exists the following short exact sequence of rings
\begin{align}
I \longrightarrow  \mathbb{C}[\operatorname{End}(V)^2]  \longrightarrow \mathbb{C}[\mathrm{OSp}(V)^2].
\end{align}
Since taking invariants by $\mathrm{OSp}(V)$ is exact, by \cref{fixedexact}, we obtain the short exact sequence
\begin{align}
I^{\operatorname{OSp}(V)} \longrightarrow  \mathbb{C}[\operatorname{End}(V)^2]^{\operatorname{OSp}(V)}  \longrightarrow \mathbb{C}[\mathrm{OSp}(V)^2]^{\operatorname{OSp}(V)}.
\end{align}
By \cref{FFT} and \cref{multilinear-invariants} we know that the polynomial invariants are given by even contractions $B(v_i,v_j)$, hence the quotient ring $\mathbb{C}[\mathrm{OSp}(V)^2]^{\mathrm{OSp}(V)}$ is even. 
\end{proof}

There have been results (see e.g. \cite{penner2019decorated}), where in the real analytic case one can prove that this space is actually of the expected real dimension $6|4$. The reason it fails in the algebraic context is that the geometric invariant quotient is too coarse. There exists a subgroup $\mathbb{Z}_2$ inside $\mathrm{OSp}(1|2)$ which acts by $-1$ on the fermionic coordinates. To be more explicit, by \cref{invariants-are-supertraces}, we know all invariants on the space of matrices are given by supertraces. We recall that the supertrace is a function taking values in an even vector space:
\begin{align}
    \operatorname{str}: \operatorname{OSp}(V) \longrightarrow \mathbb{C}^{1|0}
\end{align}
and thus we cannot determine any odd generators with the help of traces. It is also a consequence of the fact that there is a natural $\mathbb{Z}_2$-action acting on the odd coordinates. Explicitly, it is 
\begin{align*}
    \begin{pmatrix*}
        -1 & 0 & 0 \\
        0 & -1 & 0 \\
        0 & 0 & 1 
    \end{pmatrix*} 
    \begin{pmatrix*}
        1 & 0& \xi \\
        0 & 1 & 0 \\
        0 & -\xi & 1
    \end{pmatrix*}
    \begin{pmatrix*}
        -1 & 0 & 0 \\
        0 & -1 & 0 \\
        0 & 0 & 1 
    \end{pmatrix*} 
    =  \begin{pmatrix*}
        1 & 0& -\xi \\
        0 & 1 & 0 \\
        0 & \xi & 1
    \end{pmatrix*}.
\end{align*}
 One also sees by this calculation that taking the full ring of invariants can only give you a non-reduced ordinary scheme, since the odd degrees of freedom are only invariant up to sign, which forces them to be paired together to form an invariant object.
 
Let us now demonstrate that there exists an algebraically independent supertrace not generated by the traces already known from the Fricke-Klein case, i.e. it is not generated by $\mathrm{str}(A),\mathrm{str(B)}$ and $\mathrm{str}(AB)$. Before we start, we will need the following technical lemma. Let $\Lambda^i$ be the free odd superalgebra in $i$-generators. We will work in $\Lambda$-points, where $\Lambda \coloneqq \underset{i}{\mathsf{colim}}  \Lambda^i$. 
\begin{lemma}\label[lemma]{triaOSP}
    Let $B' \in SL(2,\mathbb{C}) \subset \operatorname{OSp}(1|2)$ and  $A_0 \in \operatorname{OSp}(1|2)$. Furthermore, assume that $B'$ can be diagonalized as $B_0=\mathrm{diag}(\mu,\mu^{-1},1)$ and $\mu\neq \mu^{-1}$ and that $A_0$ is generic, then there exists an element $g \in \operatorname{OSp}(1|2)$ that acts by simultaneous conjugation on the pair $(B_0,A_0)$ such that
    \begin{align*}
         gB_0g^{-1} =  
        \begin{pmatrix}
            \mu & 0 & 0\\
            1 & \mu^{-1} & 0\\
            0 & 0 & 1
            \end{pmatrix}  
            \begin{pmatrix}
                1 & 0& 0 \\
                0 & 1 & \psi \\
                \psi & 0 & 1
            \end{pmatrix}  \quad, \quad  
            gA_0g^{-1} =  
            \begin{pmatrix}
            \lambda & \kappa & 0\\
            0 & \lambda^{-1} & 0\\
            0 & 0 & 1
            \end{pmatrix}
        \begin{pmatrix}
            1 & 0& \xi \\
            0 & 1 & 0 \\
            0 & -\xi & 1
        \end{pmatrix} 
    \end{align*}
\end{lemma}
\begin{proof}
    Let
\begin{align*}
    g=
    \begin{pmatrix*}
        0 & -x^{-1} & 0 \\
        x & y & \nu \\
        0 & -x^{-1}\nu & 1 
    \end{pmatrix*},
\end{align*}
where $x,y, \nu$ are indeterminates and conjugate $diag(\mu, \mu^{-1},1)$, which yields 
\begin{align*}
    \begin{pmatrix*}
        \mu^{-1} & 0 & 0 \\
        xy(\mu-\mu^{-1}) & 
        \mu & \nu (-\mu + 1) \\
        \nu (\mu^{-1}-1) & 0 & 1
    \end{pmatrix*}
\end{align*}
define $ \psi \coloneqq \nu(\mu^{-1}-1) $ and choose $xy=(\mu - \mu^{-1})^{-1}$ and we can reduce this to 
\begin{align*}
   B= gB_0g^{-1}= \begin{pmatrix*}
        \mu^{-1} & 0 & 0 \\
        1 &  \mu & \psi\mu \\
        \psi & 0 & 1
    \end{pmatrix*}.
\end{align*}
On the other hand, we have to conjugate the matrix $A_0$:
\begin{align*}
    gA_0g^{-1}= 
    \begin{pmatrix*}
        0 & -x^{-1} & 0 \\
        x & y & \nu\\
        0 & -x^{-1}\nu & 1 
    \end{pmatrix*} 
    \begin{pmatrix*}
        a & b & \alpha\\
        c & d & \beta \\
        \gamma & \delta & f
    \end{pmatrix*}
    \begin{pmatrix*}
        y & -x^{-1} & -x^{-1}\nu \\
        x & 0 &  0\\
       -\nu & 0& 1 
    \end{pmatrix*}
\end{align*}
Furthermore, we write
\begin{align*}
   A= gA_0g^{-1} = 
    \begin{pmatrix*}
        S & T & \epsilon \\
        P & Q & \eta \\
        S\eta - P\epsilon & T \eta - Q \epsilon & 1 + \epsilon \eta
    \end{pmatrix*},
\end{align*}
where $P=-bx^2+cy^2+axy-dxy+\alpha\nu + y \beta \nu$.
Since we did not make a choice of $y$ yet, we can set $P=0$, and can find a solution to this polynomial. This reduces our problem to
\begin{align*}
    A=
    \begin{pmatrix*}
         S & T & \epsilon \\
        0 & Q & \eta \\
        S\eta & T \eta - Q \epsilon & 1 + \epsilon \eta
    \end{pmatrix*},
\end{align*}
 where we notice that $\eta$ is a function of $\beta$, so we work with the object $\eta(\nu)$. The function is linear in $\nu$, so we can solve for $\nu$ and due to the upper triangular shape we also rewrite $S= \lambda$, $T = \kappa $ and $Q = \lambda^{-1}$ and we arrive at
 \begin{align*}
    A=
    \begin{pmatrix*}
        \lambda & \kappa & \epsilon \\
        0 & \lambda^{-1} & 0 \\
        0 & \lambda^{-1} \epsilon & 1
    \end{pmatrix*}.
 \end{align*}
Finally, define $\xi \coloneqq \lambda^{-1}\epsilon$ can write down the normal form
\begin{align*}
    \begin{pmatrix*}
        \lambda & \kappa & \lambda \xi \\
        0 & \lambda^{-1} & 0 \\
        0 & \xi & 1
    \end{pmatrix*}.
\end{align*}
 We can now write down the invariant, up to sign, odd functions $\theta, \xi$. It remains to determine $\nu$. In the conjugation one computes that 
\begin{align*}
    \epsilon &= -x^{-1}\beta \\
    \eta &= -x^{-1}(xa+yc)\beta + (x\alpha + y\nu+\beta f) \\
\end{align*}
setting $\eta = 0$ one can solve for $\nu$ which yields 
\begin{align*}
    \nu &= y^{-1}(x^{-1}(xa+yc)\beta - x\alpha - \beta f ).
\end{align*}
Collecting all of the previous results we can write down 
\begin{align*}
    \psi&= y^{-1}(x^{-1}(xa+yc)\beta - x\alpha - \beta f )(\mu^{-1}-1) \\
    \xi &= -\lambda^{-1} x^{-1}\beta,
\end{align*}
where we remind ourselves that $x= y^{-1}(\lambda - \lambda^{-1})^{-1}$ and $y$ is a solution of the polynomial $P$, which exists by \cite[Lemma 2.1.]{kobayashi1988eigenvalues}, which states that the roots of a polynomial $f \in \Lambda_0[x]$ are in one to one correspondence with the roots of its bosonic reduction.
\end{proof}
Let us also discuss the analogous example from the $\mathrm{SL}_2$-case.
\begin{example}[$\mathrm{tr(ABA)}$]
Let $A, B \in \mathrm{SL}_{2}(\mathbb{C})$. We can compute the polynomial expression of $\mathrm{tr}(ABA)= \mathrm{tr}(A^2B)$ using \cref{triaOSP} as follows.
Let $A$ be upper triangular and $B$ lower triangular, then using the notation from \cref{triaOSP}, we have
\begin{align}
    \begin{split}
        \mathrm{tr}(A) &= \lambda + \lambda^{-1} \\
        \mathrm{tr}(B) & = \mu + \mu^{-1} \\
        \mathrm{tr}(AB) & = \lambda\mu + \lambda^{-1} \mu^{-1} + \kappa.
    \end{split}
\end{align}
Similarly, one can compute 
\begin{align}\label{trAB}
    \mathrm{tr}(A^2B) = \lambda^2 \mu + \kappa(\lambda + \lambda^{-1}) + \mu^{-1} \lambda^{-2}.
\end{align}
One notices that $\mathrm{tr}(A^2B)$ can be rewritten as 
\begin{align}
    \mathrm{tr}(A^2B) = \mathrm{tr}(AB)\mathrm{tr}(A) - \mathrm{tr}(B).
\end{align}
\end{example}

On the other hand, one can show that for $A,B \in \mathrm{OSp}(1|2)$ the supertrace $\mathrm{str}(ABA)$ cannot be written in terms of $\mathrm{str}(A),\mathrm{str(B)}$ and $\mathrm{str}(AB)$.
\begin{proposition} \label{atleast4}
    Consider $\operatorname{OSp}(1|2)$ as an algebraic supergroup over $\mathbb{C}$ then $\operatorname{str}(A^2B)$ is not an element of the ring generated by $\operatorname{str}(A)$, $\operatorname{str}(B)$ and $\operatorname{str}(AB)$.
\end{proposition}
\begin{proof}
    Using \cref{triaOSP}, we compute 
\begin{align*}
    A^2B&= \begin{pmatrix}
        \lambda & \kappa & \xi \lambda\\
        0 & \lambda^{-1} & 0\\
        0 & -\xi & 1
    \end{pmatrix}
    \begin{pmatrix}
    \lambda & \kappa & \xi \lambda\\
    0 & \lambda^{-1} & 0\\
    0 & -\xi & 1
\end{pmatrix}
\begin{pmatrix}
    \mu & 0 & 0\\
    1 & \mu^{-1} & \psi \mu^{-1}\\
    \psi & 0 & 1
\end{pmatrix}\\
&=\begin{pmatrix}
    \lambda^2 & \kappa(\lambda + \lambda^{-1}) & \xi \psi \lambda^2+\xi \lambda \\
    0 & \lambda^{-2} & 0\\
    0 & -\xi\lambda^{-1} - \xi & 1
\end{pmatrix} 
\begin{pmatrix}
    \mu & 0 & 0\\
    1 & \mu^{-1} & \psi \mu^{-1}\\
    \psi & 0 & 1
\end{pmatrix}\\
&=\begin{pmatrix}
    \lambda^2\mu +  \kappa(\lambda + \lambda^{-1}) + \xi \psi \lambda ( \lambda+ 1)&  &  \\
     & \lambda^{-2}\mu^{-1} & \\
     &  & -\xi \psi \lambda^{-1} \mu^{-1}- \xi \psi \mu^{-1} + 1
\end{pmatrix}.
\end{align*}
We note, that we have only explicitly written the diagonal elements, the other entries are also non-zero but are irrelevant for this discussion. It follows that the supertrace is
\begin{align}
    \operatorname{str}(A^2B)=\lambda^2 \mu + \kappa (\lambda + \lambda^{-1}) + \lambda^{-2}\mu^{-1} + \xi \psi (+\lambda(\lambda+1)+\mu^{-1}(\lambda^{-1}+1))-1.
\end{align}
We observe, that the bosonic part of this expression is exactly \eqref{trAB}. Therefore, using \cref{triaOSP}, we compute
\begin{align*}
    &\quad \  \mathrm{str}(A^2B)+1-((\mathrm{str}(A)+1)(\mathrm{str}(AB)+1)-(\mathrm{str}(B)+1)) \\
    &=  \xi \psi (+\lambda(\lambda+1)+\mu^{-1}(\lambda^{-1}+1)) - \xi \psi (\lambda + \mu^{-1})(\lambda + \lambda^{-1})\\
    &= \xi \psi (\lambda + \mu^{-1} - 1 - \mu^{-1}\lambda) \\
    & = \xi \psi (\lambda - 1+\mu^{-1}(1-\lambda) ) \\
    &= \xi \psi ( \mu^{-1}-1)(1-\lambda)
\end{align*}
From this, one can see the $\xi \psi$ term cannot be generated in terms of the previously mentioned supertraces.
\end{proof}
\begin{remark}
    In upcoming work, we will use this supertrace to construct Fenchel--Nielsen coordinates for super Teichmüller space in an analogous manner to \cite{figueroa2025fenchelnielsencoordinatessl3crepresentations}.
\end{remark}
\section{Geometric Interpretation}
\subsection{Character Variety}
We follow \cite{sikora2012character}. Let $G$ be a reductive algebraic group over k. When looking at the set $\mathrm{Hom}(\Gamma,G)$, we can associate to it a ring $\mathcal{R}(\Gamma,G)$, and the $k$-points of $\operatorname{Spec}(\mathcal{R}(\Gamma,G))|= \operatorname{Hom}(\Gamma,G)$. Since every finitely generated group can be constructed by imposing a set of relations on the free group $F_n$, we will start by constructing the representation variety of $F_n$.
\begin{proposition}
    Let $G$ be an algebraic group over $k$ and $F_n$ the free group on $n$ letters then coordinate ring of the representation variety is $\mathcal{R}(F_n,G) = k[G]^{\otimes n}$.
\end{proposition}
\begin{proof}
    It is clear, that a morphism of groups between $F_n$ and $G$ is determined by the images of the generators, there are no relations and thus a representation of $F_n$ is simply a tuple of $n$ elements of $G$, hence the coordinate ring the representation variety is given by
    \begin{align}
        \mathcal{R}(F_n,G) = k[G]^{\otimes n}
    \end{align}
\end{proof}
If we now impose relations on the free group, we get 
\begin{align*}
   \operatorname{Hom}(\Gamma,G)=\operatorname{Spec}( k[G \times ... \times G]/ I_\Gamma).
\end{align*}
The ideal $I_\Gamma$ is generated by the relations induced by $\Gamma$, which tells us that the images of the generators must satisfy the relations of $\Gamma$, because otherwise it would not be a group homomorphism.
By the quotient nature of $R(\Gamma,G)$ it is not obvious if it is an integral domain, which over $\mathbb{C}$ allows you to identify Specm and Spec.
This space algebraically classifies all so-called $G$-representations. However, representations are only meaningful up to conjugation. Thus, we define 
\begin{definition}
    Let $G$ be a reductive algebraic group over $k$ and $\Gamma$ a finitely generated group, then we call the categorical quotient
    \begin{align}
        \operatorname{Spec}(R(\Gamma,G)^G)
    \end{align}
    the character variety.
\end{definition}
Observe that there are no obstructions to replacing the algebraic group with an algebraic supergroup.
\begin{remark}
    One observes that this is a purely Hopf-algebraic construction, hence one can also generalize this question to quantum groups. A construction was already given in the physics literature \cite{alekseev1996quantum}.
\end{remark}
\subsection{Character Stack}
In view of \cref{main}, which seems to suggest that geometric invariant theory might not be well-suited when considering problems in supergeometry, one may consider studying the character stack of an algebraic supergroup instead. We give a brief account on a gradedness result of the character stack of $\mathrm{OSp}(1|2)$, and we prove that the character stack has the expected dimension of $3|2$. By $\infty$-stack, we mean the full subcategory of sheaves of anima (also known as $\infty$-groupoids) with respect to the \'etale topology, and we will refer to it as stack, if the sheaf factors through the category of groupoids. Furthermore, it is a localization of the category of presheaves:
\begin{align}
   L:\mathsf{Fun} (\mathsf{sRings},\mathsf{Ani}) \adjoint \mathsf{Stk} : \iota.
\end{align}
Let $X$ be a Riemann surface and consider the representation variety 
\begin{align}
    \mathrm{Rep}(X)\coloneqq \operatorname{Hom}(\pi_1(X),\operatorname{OSp}(1|2))
\end{align}
and its stack quotient $\mathfrak{X}\coloneqq[\mathrm{Rep}(X)/\operatorname{OSp}(1|2)]$, called the character stack. For more details on the theory of superstacks, we refer to \cite{bruzzo2025foundations}. We investigate whether this space is graded, which we are going to define now.
\begin{definition}
    Let $X$ be a smooth superstack and $\mathcal{J}$ be the sheaf generated by the nilpotents, where we denote $\mathcal{E}\coloneqq \mathcal{J}/\mathcal{J}^2$. A superstack $X$ is called \textit{graded}, if and only if there exists an isomorphism such that 
    \begin{align}
        \mathcal{O}_X \simeq \wedge^\bullet \mathcal{E},
    \end{align}
\end{definition}
The question whether a space is graded or not depends on a classical obstruction theory argument, i.e. does a certain cohomology class vanish on the underlying space \cite{voronov1990elements} \cite{manin1997gauge}. We recall that for the stack quotient $[X/G]$ of a superscheme $X$, we have an equivalence of categories
\begin{align}
    \operatorname{QCoh}([X/G])\simeq \operatorname{QCoh}^G(X).
\end{align}
Furthermore, for $X \in \operatorname{sAff}$ and $\mathcal{F} \in \operatorname{QCoh}(X)$, $H^i(X,\mathcal{F}) = 0 \ \forall i >0$, which is equivalent to the global section functor being exact.
\begin{remark}
    This property of the global section functor being exact is called \textit{cohomologically affine}.
\end{remark}
\begin{proposition}\label[proposition]{cohomologicallyaffine}
    Let $k$ be an algebraically closed field of characteristic $0$, $G$ a reductive algebraic supergroup, and $X$ an affine $k$-superscheme. Then 
    \begin{align}
        H^i_G(X,\mathcal{F}) = 0 \quad \forall i > 0.
    \end{align}
\end{proposition}

\begin{proof}
    The global sections functor $\Gamma$ is exact, since $X$ is affine and $(-)^G$ is exact, by \cref{fixedexact}. $G$-equivariant sheaf cohomology is defined to be the right derived functor of the composition $(-)^G \circ \Gamma$, thus by composition of exact functors, the composition is also exact. We obtain 
\begin{align}
    H^i_G(X,\mathcal{F}) \subset H^i(X,\mathcal{F}) = 0 \ \forall i > 0,
\end{align}
as desired.
\end{proof}
\begin{corollary}
    Let $X \in \operatorname{sSch}$ and $G \in \operatorname{Grp(sSch)}$. If $X$ is cohomologically affine and smooth and the fixed point functor $(-)^G$ is exact, then the stack quotient $[X/G]$ is graded.
\end{corollary}
\begin{proof}
    Let $\mathcal{J}$ be the sheaf of nilpotents on $[X/G]$ and $\mathcal{E}\coloneqq \mathcal{J}/\mathcal{J}^2$. Furthermore, we denote $\mathcal{F}^i \coloneqq T_{(-)^i}[X/G]_{\operatorname{red}} \otimes \wedge^i \mathcal{E}^\vee$. To prove $[X/G]$ is graded, we need to show the that obstruction classes
    \begin{align}
        \omega_i \in H^1([X/G]_{\operatorname{red}},\mathcal{F}^i)
    \end{align}
    vanish \cite{voronov1990elements}. Given the canonical inclusion $\iota: [X/G]_{\operatorname{red}} \rightarrow [X/G]$ one can pushforward the sheaves along the inclusion and one can make the following identifications
    \begin{align}
        H^1([X/G]_{\operatorname{red}},\mathcal{F}^i) = H^1([X/G],\iota_*\mathcal{F}^i) = H^1_G(X,\iota_*\mathcal{F}^i) = 0,
    \end{align}
    where the last identification follows from \cref{cohomologicallyaffine}. Thus, all obstruction classes vanish, so $[X/G]$ is graded, as desired.
\end{proof}
As a consequence of gradedness, the superstack can also be written in the following form:
\begin{align}
    [X/G] \simeq ([X/G]_{\operatorname{red}}, \wedge^\bullet \mathcal{E}),
\end{align}
where $\mathcal{E} \coloneqq \mathcal{J}/\mathcal{J}^2$. As a consequence, we get an easier description of the Berezinian bundle 
\begin{align}
    \operatorname{Ber}(\mathcal{T}_{[X/G]_{sm}}) = \operatorname{det} (\mathcal{T}_{[X/G]_{sm}})\otimes \operatorname{det}^{-1}(\mathcal{E})
\end{align}
\begin{corollary}
    The character stack $\mathfrak{X}(F_n,\mathrm{OSp}(1|2))$ is graded.
\end{corollary}
Lastly, given a point $x: \mathrm{Spec}(\mathbb{C}) \rightarrow \mathfrak{X}$, we can consider the tangent complex at the point $x$, which is given by the two-term complex 
\begin{align}
    [\mathfrak{g} \longrightarrow \mathcal{T}_{X,x}].
\end{align}
Along the smooth locus, we obtain the tangent space 
\begin{align}
    \mathcal{T}_{X,x}/\mathfrak{g}.
\end{align}
This object will have the expected dimension of $(2g-2)\mathrm{dim}(\mathrm{OSp}(1|2))$ and for the details on the explicit calculations we will refer the interested reader to \ref{tangent}. 
\appendix
%\section{Elementary Proof of Fricke-Klein}
%\input{ElementaryFricke}
\section{Computations}
This appendix provides a detailed derivation of the dimension of the tangent space of the representation variety $\mathcal{T}_{X,x}$. We conclude that the character stack has the expected dimension for a super Riemann surface of genus $g$. The classical proofs can be adapted with minimal modification from the literature, which we will be pointing towards for each statement. The necessary modification one has to do to adapt these results to the supergeometric case is replace the usual dual numbers, by the so-called super dual numbers, which refers to the superalgebra $k[\varepsilon,\eta] \coloneqq k[X,\Theta]/(X^2,X \Theta)$, where $\mathrm{deg(\varepsilon)= 0}$ and $\mathrm{deg}(\eta)= 1$. We begin by proving the following lemma.
\begin{lemma}[\cite{milne2017algebraic}]
    Let $(A, \Delta)$ be a super Hopf algebra over k and let $I$ denote the kernel of the co-identity map $\epsilon$, then $A \simeq k \oplus I$ as k-modules.
\end{lemma}
\begin{proof}
    The co-identity map is a morphism of super algebras, so in particular it is a morphism of super vector spaces, hence we get the following sequence
    \begin{center}
        \begin{tikzcd}
        0 \arrow[r] & \mathrm{ker}(\epsilon) \arrow[r] & A \arrow[r, "\epsilon"] & k \arrow[r] & 0 
        \end{tikzcd}
    \end{center}
    This is a sequence of super vector spacecs, so it splits and we get the desired isomophism.
\end{proof}
The next theorem was originally proven in \cite{lubotzky1985varieties}. However, we will follow the more detailed presentation of \cite{sauer2020character}.
\begin{theorem}\label{tangent}
    Let $R(\Gamma,G)$ be the coordinate ring of the representation variety $\mathrm{Hom}(\Gamma,G)$ for $\Gamma$ a finitely generated group and $G$ an algebraic supergroup. Then the Zariski tangent space at a point $\rho$ is isomorphic to $Z(\Gamma, Ad_\rho)$.
\end{theorem}
\begin{proof}
    We can now describe a morphism from the supergroup $G(\mathbb{C}[\varepsilon,\eta])$ to the Lie algebra $\mathfrak{g}$.
Define 
\begin{align*}
    \sigma(g) = g \pi(g)^{-1}
\end{align*} 
where $\pi$ is a morphism of supergroups $G(\mathbb{C}[\varepsilon,\eta]) \rightarrow G(\mathbb{C})$ induced by $\mathbb{C}[\varepsilon,\eta]\rightarrow \mathbb{C}$, where we send both variables to zero.
By applying $\pi$ again, we can check if it is an element of the Lie algebra:
\begin{align*}
    \pi(\sigma(g))  & = \pi ((g \pi(g)^{-1})) \\
                    & = \pi(g) \pi^2(g^{-1}) \\
                    & = e
\end{align*}
We can now compute the following
\begin{align*}
    \sigma(g_1g_2) & = g_1g_2\pi(g_1g_2)^{-1} \\
                   & = g_1g_2\pi(g_2)^{-1}\pi(g_1)^{-1} \\
                   & = g_1(\pi(g_1)^{-1}\cdot \pi(g_1))\sigma(g_2)\pi(g_1)^{-1} \\
                   & = \sigma(g_1) \cdot Ad_{\pi(g_1)} \sigma(g_2)
\end{align*}
And we are allowed to switch to additive notation and we get the 1-cocycle condition for group cohomology
\begin{align*}
    \sigma(gh) = \sigma(g) + Ad_{\pi(g)} \sigma(h)
\end{align*}
So it is a cycle in $Z^1(\Gamma, \mathfrak{g})$. \\
We now also get a map from $\Gamma \rightarrow \mathfrak{g}$, by using the universal property
\begin{center}
    \begin{tikzcd}
        \Gamma \arrow[r, "\rho_U"] \arrow[dr, "\psi_f"] & G(R(\Gamma,G)) \arrow[d, "G(f)"] \\
         & G(\mathbb{C}[\varepsilon, \eta]) \arrow[d, "\sigma"]\\
         & \mathfrak{g}
    \end{tikzcd}
\end{center}
We will call this map $\Sigma_f \coloneqq \sigma \circ \psi_f$, which fulfills the cocycle condition, because $\psi_f$ is a group homomorphism and the properties of $\sigma$. Consider the diagram 
\begin{center}
    \begin{tikzcd}
        \Gamma \arrow[r, "\rho_U"] \arrow[ddr, "\rho", swap] & G(R(\Gamma,G)) \arrow[d, "G(f)"] \arrow[dd, bend left=60, "G(h_\rho)"] \\
        & G(\mathbb{C}[\varepsilon, \eta]) \arrow[d] \\
        & G(\mathbb{C})
    \end{tikzcd}
\end{center}
Let $f \in T_\rho \mathrm{Hom}(\Gamma,G) = \{f| R(\Gamma, G) \xrightarrow{f} k[\varepsilon, \eta] \xrightarrow{\pi} k, \pi \circ f = h_\rho \}$ then we get a morphism 
\begin{align*}
    \Psi_\rho: \ T_\rho \mathrm{Hom}(\Gamma, G) & \longrightarrow Z^1(\Gamma, \mathfrak{g}_{Ad_\rho}) \\
    f &\longmapsto \Sigma_f
\end{align*} 
Now let us also define an inverse. Let $\Sigma \in Z^1(\Gamma, \mathfrak{g}_{Ad_\rho})$
\begin{align*}
    \phi_\Sigma: \ \Gamma & \longrightarrow G(\mathbb{\varepsilon, \eta}) \\
    \gamma & \longmapsto \Sigma (\gamma) \cdot \rho(\gamma)
\end{align*}
We calculate:
\begin{align*}
    \phi_\Sigma(\gamma \delta) & = \Sigma(\gamma \delta) \cdot \rho(\gamma) \\
    & = \Sigma(\gamma) \cdot Ad_{\rho(\gamma)}\Sigma(\delta) \cdot \rho(\gamma \delta) \\
    & = \Sigma(\gamma) \rho(\gamma) \Sigma(\delta) \rho(\gamma)^{-1} \rho(\gamma) \rho(\delta) \\
    & = \Sigma(\gamma)\rho(\gamma) \Sigma(\delta) \rho(\delta) \\
    & = \phi_\Sigma (\gamma) \phi_\Sigma (\delta)
\end{align*}
So this is indeed a group homomorphism. By the universal property we get a morphism 
\begin{center}
    \begin{tikzcd}
        \Gamma \arrow[r, "\rho_U"] \arrow[rd, "\phi_\Sigma", swap] & G(R(\Gamma, G)) \arrow[d, "G(f_\Sigma)"]\\
         & G(\mathbb{C}[\varepsilon, \eta])
    \end{tikzcd}
\end{center}
Recall that $\mathfrak{g} \simeq ker(\pi)$ and $\rho(\gamma) \in G(\mathbb{C})$. It follows that $\pi \circ \phi_\Sigma = \rho$, because
\begin{align*}
    \pi (\phi_\Sigma(\gamma)) = \pi(\Sigma(\gamma \cdot \rho(\gamma))) = \rho(\gamma).
\end{align*}
So, we get the map 
\begin{align*}
   \Phi: \  Z^1(\Gamma, \mathfrak{g}_{Ad_\rho}) & \longrightarrow T_\rho \mathrm{Hom}(\Gamma, G) \\
\Sigma & \longmapsto f_\Sigma
\end{align*}
The only thing remaining is that we have to prove, that $\Psi_\rho, \Phi_\rho$ are inverse to each other. Applying $\Phi_\rho(\Psi_\rho(f)) =f_{\Sigma_f}$. By the universal property 
\begin{center}
    \begin{tikzcd}
        \Gamma \arrow[r, "\rho_U"] \arrow[dr, "\phi_{\Sigma_f}", swap] & G(R(\Gamma, G)) \arrow[d, "G(f_{\Sigma_f})"] \\
        & G(\mathbb{C}[\varepsilon, \eta])
    \end{tikzcd}
\end{center}
we can now calculate 
\begin{align*}
    \phi_{\Sigma_f}(\gamma) &= \Sigma_f(\gamma) \cdot \rho(\gamma) \\
                            &= (\sigma \circ G(f)\circ \rho_U)(\gamma) \cdot \rho(\gamma )\\
                            &= (G(f) \circ \rho_U(\gamma)) \cdot ((\pi \circ G(f)\circ \rho_U)(\gamma))^{-1} \cdot \pi \circ G(f)\circ \rho_U)(\gamma) \\
    \phi_{\Sigma_f}(\gamma)&= G(f) \circ \rho_U(\gamma)
\end{align*}
By uniqueness of the universal property, we get $\phi_{\Sigma_f} = \psi_f$. It follows that $G(f) = G(f_{\Sigma_f})$ so we get $f = f_{\Sigma_f}$. The other direction goes as follows:
\begin{align*}
    (\Psi_\rho \circ \Phi_\rho)(\Sigma) & = \sigma \circ \psi_{f_\Sigma} \\
    &= \sigma \circ G(f_{\Sigma}\circ \rho_U) \\
    &=(G(f_{\Sigma}) \circ \rho_U)(\gamma)((\pi \circ G(f_{\Sigma}) \circ \rho_U)(\gamma))^{-1} \\
    &= \phi_\Sigma(\gamma) \cdot \rho(\gamma) (\pi(\phi_\Sigma(\gamma)))^{-1} \\
    &= \Sigma(\gamma) \cdot \rho(\gamma) (\pi(\Sigma(\gamma)\cdot \rho(\gamma)))^{-1} \\
    &=\Sigma(\gamma) \cdot \rho(\gamma) \cdot e \cdot \rho(\gamma)^{-1}\\
    &= \Sigma(\gamma) 
\end{align*}
So indeed the two morphisms are inverse to each other and we get the isomorphism
\begin{align*}
    Z^1(\Gamma, \mathfrak{g}_{Ad_\rho}) \simeq T_\rho \mathrm{Hom}(\Gamma, G)
\end{align*}
\end{proof}
The next proposition, was originally proven in \cite{goldman1984symplectic} and we will follow the explicit exposition given in \cite{notes-character-varieties}.
\begin{proposition}\label[proposition]{Arnaud}
    Let $X$ be a Riemann surface of genus $g$ and $G \in \mathsf{Grp(Sch)}$, such that it admits a symmetric, bilinear, non-degenerate pairing. Then the dimension of the 1-cocycles is given by 
    \begin{align}
        \mathrm{dim}Z^1(\Gamma,Ad_\rho) = (2g-1)\mathrm{dim}(G) + \mathrm{dim}(C_{G}(\pi_1(X))).
    \end{align}
\end{proposition}
\begin{proof}
    We give a sketch of the proof by providing the key arguments, which should convince the reader. First of all, since $X$ is an Eilenberg-MacLane space $K(\pi_1(X),1)$ computing group cohomology is the same as computing de Rham cohomology with local coefficients, which lets us identify 
    \begin{align}
        H^*(\pi_1(X),Ad_\rho) \simeq H^*_{dR}(X,E_\rho),
    \end{align}
    so the cohomology vanishes in degrees larger than $2$. Furthermore, we have the Euler characteristic formula,
    \begin{align}
       \chi(X,\rho)  =\mathrm{dim}(H^0(\pi_1(X),Ad_\rho)) - \mathrm{dim}(H^1(\pi_1(X),Ad_\rho)) + \mathrm{dim}(H^2(\pi_1(X),Ad_\rho)).
    \end{align}
    This is independent of $\rho$, because when defining group cohomology, the twisting by $\rho$ is only realized inside of the differentials of the chain complex, so the chain complex itself does not see the twisting. For the trivial representation $\rho$ the Euler characteristic turns into 
    \begin{align}
        \chi(X,\mathrm{triv}) = \chi(X) \mathrm{dim}(G).
    \end{align}
    By Poincaré duality and the non-degeneracy of the Killing form, we obtain 

    \begin{align}
        H^2(\pi_1(X),Ad_\rho) \overset{PD}{\simeq} H^0(\pi_1(X),Ad_\rho^\vee)^\vee \overset{Ad_\rho^\vee \simeq Ad_\rho}{\simeq} H^0(\pi_1(X),Ad_\rho).
    \end{align}
    It is true, that $H^0(\pi_1(X),Ad_\rho)$ is the space of $Ad(\rho)$-invariant elements and also that $\mathrm{dim}(B^1(\pi_1(X),Ad_\rho)) = \mathrm{dim}(G)-\mathrm{dim}(C(\rho))$. Combining these two facts we obtain
    \begin{align}
        \begin{split}
            \mathrm{dim}(Z^1(\pi_1(X),Ad_\rho))&= \mathrm{dim}(H^1(\pi_1(X),Ad_\rho)) + \mathrm{dim}(B^1(\pi_1(X),Ad_\rho)) \\
            & = \chi(X) \mathrm{dim}(G) + \mathrm{dim}(G) - \mathrm{dim}(C(\rho)) \\
            &= (2g-2) \mathrm{dim}(G) + \mathrm{dim}(G) - \mathrm{dim}(C(\rho)) \\
            &=(2g-1)\mathrm{dim}(G) + \mathrm{dim}(C(\rho))
        \end{split}
    \end{align}
    We note that $\mathrm{dim}( C(\rho)) = 
    \mathrm{dim}(C(G))$ minimizes the dimension. 
 \end{proof}
 As a corollary we obtain the desired dimension of the tangent space.
 \begin{corollary}
    The dimension of the tangent space of a smooth point of the representation variety $\mathrm{Rep}(\pi_1(X), \mathrm{OSp}(1|2))$ is $(2g-1)\mathrm{dim}(\mathrm{OSp}(1|2))$.
 \end{corollary}
 \begin{proof}
     A point is smooth if the dimension of the scheme and the Zariski tangent space coincide. Furthermore, the center of $\mathrm{OSp}(1|2)$ is given by $\{ \mathrm{id}\}$, thus the dimension of the center is $0$ dimensional. By \cref{Arnaud} the claim follows.
 \end{proof}
\printbibliography
\end{document}